\documentclass{amsart}

\usepackage[all]{xy}
\usepackage{amsmath}
\usepackage{amssymb}
\usepackage{amsthm}
\usepackage{amscd}
\usepackage{enumerate}

\DeclareMathOperator{\ant}{Ant}
\DeclareMathOperator{\rad}{rad}
\DeclareMathOperator{\diam}{diam}
\DeclareMathOperator{\leaf}{Leaf}

\newcommand{\sph}{\mathbb S}
\newcommand{\snr}{\mathbb S_r^N}
\newcommand{\snrx}{\snr(x)}
\newcommand{\sgr}{\mathbb S_r^G}
\newcommand{\sgrx}{\sgr(x)}

\newcommand{\ee}{\left}
\newcommand{\ii}{\right}
\newcommand{\cart}{\, \square \,}
\newcommand{\cartgh}{G \cart H}
\newcommand{\powg}{G^N}
\newcommand{\cc}{\mathbb C}
\newcommand{\kmn}{K_m^N}
\newcommand{\antg}{\ant_G}
\newcommand{\anth}{\ant_H}
\newcommand{\unig}{\ant(G)}
\newcommand{\unih}{\ant(H)}
\newcommand{\antgx}{\antg(x)}
\newcommand{\dg}{d_G}

\newcommand{\eccx}{\epsilon(x)}
\newcommand{\radg}{\rad(G)}
\newcommand{\diamg}{\diam(G)}

\newtheorem{ex}{Example}
\newtheorem{notation}{Notation}
\newtheorem{question}{Question}

\newtheorem{theorem}{Theorem}
\newtheorem{definition}[theorem]{Definition}
\newtheorem{proposition}[theorem]{Proposition}
\newtheorem{cor}[theorem]{Corollary}
\newtheorem{lemma}[theorem]{Lemma}

\newtheorem{remark}[theorem]{Remark}

\numberwithin{theorem}{section}

\begin{document}
\title[Exponential Maximal Bounds for Trees]{Dimensionally Exponential Lower Bounds on the $L^p$ Norms of the Spherical Maximal Operator for Cartesian Powers of Finite Trees and Related Graphs}
\author{Jordan Greenblatt}
\address{Department of Mathematics, UCLA, Los Angeles, CA 90095-1555}
\email{jsg66@math.ucla.edu}
\maketitle
\vspace{-0.3in}
\begin{abstract}
Let $T$ be a finite tree graph, $T^N$ be the Cartesian power graph of $T$, and $d^N$ be the graph distance metric on $T^N$. Also let
\[
\snrx := \{v \in T^N: d^N(x,v) = r\}
\]
be the sphere of radius $r$ centered at $x$ and $M$ be the spherical maximal averaging operator on $T^N$ given by
\[
Mf(x) := \sup_{\substack{r \geq 0 \\ \snrx \neq \emptyset}} \frac{1}{|\snrx|} \ee|\sum_{\snrx} f(y)\ii|.
\]
We will show that for any fixed $1 \leq p \leq \infty$, the $L^p$ operator norm of $M$, i.e.
\[
\|M\|_p := \sup_{\|f\|_p = 1} \|Mf\|_p,
\]
grows exponentially in the dimension $N$. In particular, if $r$ is the probability that a random vertex of $T$ is a leaf, then $\|M\|_p \geq r^{-N/p}$, although this is not a sharp bound.

This exponential growth phenomenon extends to a class of graphs strictly larger than trees, which we will call \emph{global antipode graphs}. It stands in contrast to the results of \cite{HKS}, \cite{BK}, and \cite{GKK} which together prove that the spherical maximal $L^p$ bounds (for $p > 1$) are dimension-independent for finite cliques.
\end{abstract}

\section{Introduction}\label{intro}

Unless otherwise stated, all graphs will be simple, connected, undirected, unweighted, and finite and all functions will be complex valued. For a fixed graph $G$ and $x,y \in V(G)$, we use the notation $x \sim_G y$ to signify that the pair $\{x,y\}$ is in $E(G)$. We also denote by $d_G$ the "shortest distance" metric on $G$, although we sometimes suppress the subscript when the graph is clear from context. We recall the following standard definitions:

\begin{definition}[Cartesian Graph Product]
For graphs $G$ and $H$, the Cartesian product of $G$ and $H$ will be denoted by $G \cart H$. The vertex set $V(G \cart H)$ is the Cartesian set product $V(G) \times V(H)$. Two vertices $(g,h), (g',h') \in V(G \cart H)$ are adjacent in $G \cart H$ if and only if either
\begin{enumerate}
\item $g = g'$ and $h \sim_H h'$ or
\item  $h = h'$ and $g \sim_G g'$.
\end{enumerate}
\end{definition}

\begin{definition}[Cartesian Graph Power]
The $N^{th}$ Cartesian power of a graph $G$ will be denoted by $G^N$. For a positive integer $N$ we inductively define
\[
G^N := \begin{cases} G &\mbox{if } N = 1 \\ 
G^{N-1} \cart G & \mbox{if } N > 1.\end{cases} 
\]
More concretely, $V(G^N)$ is the Cartesian set power $V(G)^N$ and for two vertices $x,y \in V(G^N)$, we say $x \sim_{G^N} y$ if and only if there exists a unique index $1 \leq k \leq N$ such that
\begin{enumerate}
\item for all $j \neq k$, $x_j = y_j$ and
\item $x_k \sim_G y_k$.
\end{enumerate}
\end{definition}

\begin{remark}\label{distancerelation}
By direct computation one observes that for any $(g,h), (g',h') \in V(G \cart H)$,
\[
d_{G \cart H}\big((g,h),(g',h')\big) = d_G(g,g') + d_H(h,h').
\]
Inductively applying this computation, one observes that for any $x,y \in G^N$,
\[
d_{\powg}(x,y) = \sum_{k=1}^N d_G(x_k,y_k).
\]
\end{remark}

We view a graph $G$ as a natural metric measure space under the (integer valued) metric $d_G$ and the counting measure. We define
\[
\sgrx := \{y \in V(G): d_G(x,y) = r\}
\]
to be the radius $r$ sphere in $G$ centered at vertex $x \in V(G)$.

\begin{definition}[Spherical Maximal Operator]
The spherical maximal operator $M^G$ is defined for any $f: G \to \cc$ by
\[
M^Gf(x) := \sup_{\substack{r \geq 0 \\ \sgrx \neq \emptyset}} \frac{1}{|\sgrx|} \ee|\sum_{\sgrx} f(y)\ii|.
\]
\end{definition}

$M^G$ satisfies an immediate operator bound of $1$ on $L^\infty(G)$. Therefore, for any $1 \leq p < \infty$,
\[
\sum_{x \in G} |M^Gf(x)|^p \leq \sum_{x \in G} \|f\|^p_\infty \leq \|f\|^p_\infty |G| \leq \|f\|^p_p |G|
\]
so $M^G$ immediately satisfies an $L^p$ bound of $|G|^{1/p}$.

We will often suppress the superscript $G$. Although it is straightforward to see that $M$ satisfies an $L^p$ bound, the bound $|G|^{1/p}$ is not in general sharp. The optimal bound depends heavily on the structure of $G$. In \cite{HKS}, \cite{BK}, and \cite{GKK} the authors explored the dimensional asymptotics of the $L^p$ bounds on spherical maximal operator for Cartesian powers of finite cliques $G = K_m$ for all $m \geq 2$.

If $p < \infty$, the a priori $L^p$ bound for the spherical maximal operator on $K_m^N$ is exponential in the dimension $N$, namely $m^{N/p}$. By viewing $M$ more carefully as the maximum of $N+1$ contractions on $L^1(\kmn)$ and $L^\infty(\kmn)$, namely the spherical averaging operators, we can bound $M$ in $L^1$ by $N+1$, simply dominating $Mf$ by the sum of all the spherical averages for any $f \geq 0$. Linearizing the maximal operator as in \cite{NS} (top of p. 151) and interpolating between the $L^1$ and $L^\infty$ endpoints provides the improved bound $(N+1)^{1/p}$.

Testing $M$ on a delta mass in $\kmn$ shows that the $L^1$ bound of $N+1$ is sharp. However, the three papers collectively prove the surprising result that if $1 < p < \infty$, $M$ satisfies an $L^p$ bound depending \emph{only} on $m$ and $p$. In particular, the bound is \emph{independent of the dimension $N$}.

The author is currently investigating the full relationship between the structure of a graph and the dimensional asymptotics of its spherical maximal operator's $L^p$ norm. Here we establish a substantial class of graphs, including all tree graphs (adopting the convention that trees have at least $3$ vertices), whose spherical maximal bounds grow exponentially in dimension. For any graph $T$, we use the notation $\leaf(T) := \{t \in T : \deg(t) = 1\}$ to be the set of leaf vertices of $T$.

\begin{theorem}\label{treebound}
Let $T$ be a tree graph. Then, if $M_N$ is the spherical maximal operator on $T^N$, for any $p < \infty$,
\[
\|M_N\|_p^p \geq \ee(\frac{|T|}{|\leaf(T)|}\ii)^N.
\]
In particular, the bound grows exponentially in dimension.
\end{theorem}

\section{Global Antipode Graphs and Exponential Lower Bounds}\label{lowerbound}

Theorem \ref{treebound} will come as a corollary of an analogous result for a larger class of graphs, which we call \emph{global antipode graphs}. We begin with some necessary definitions.

\begin{definition}[Antipode]
Let $G$ be a graph and $x \in V(G)$ be a vertex. The antipode of $x$ in $G$, denoted $\antgx$, is the (nonempty) sphere centered at $x$ of maximal radius with respect to $\dg$.
\end{definition}

\begin{ex}[Examples of Antipodes]
In the graphs below, the vertices in the antipode of vertex $v$ are marked by empty circles:\\

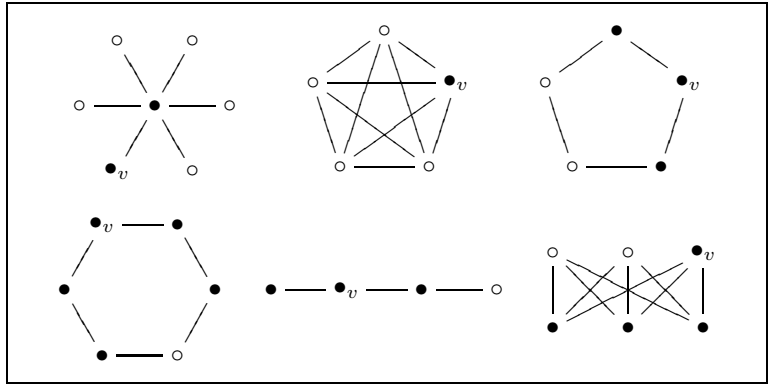
\begin{center}
\framebox[.8\textwidth]{
\parbox{.8\textwidth}{
\begin{center}
\hspace{0in}\xygraph{
!{<0cm,0cm>;<1cm,0cm>:<0cm,1cm>::}
!{(0,0) }*+{\bullet}="a"
!{(1,0) }*+{\circ}="b"
!{(.5,.87) }*+{\circ}="c"
!{(-.5,.87)}*+{\circ}="d"
!{(-1,0)}*+{\circ}="e"
!{(-.5,-.87)}*+{\bullet_{v}}="f"
!{(.5,-.87)}*+{\circ}="g"
"a"-"b" "a"-"c"
"a"-"d" "a"-"e"
"a"-"f" "a"-"g"
}\qquad
\xygraph{
!{<0cm,0cm>;<1cm,0cm>:<0cm,1cm>::}
!{(.95,.31) }*+{\bullet_v}="a"
!{(.59,-.81) }*+{\circ}="b"
!{(-.59,-.81) }*+{\circ}="c"
!{(-.95,.31)}*+{\circ}="d"
!{(0,1)}*+{\circ}="e"
"a"-"b" "a"-"c" "a"-"d" "a"-"e"
"b"-"c" "b"-"d" "b"-"e"
"c"-"d" "c"-"e"
"d"-"e"
}\qquad
\xygraph{
!{<0cm,0cm>;<1cm,0cm>:<0cm,1cm>::}
!{(.95,.31) }*+{\bullet_v}="a"
!{(.59,-.81) }*+{\bullet}="b"
!{(-.59,-.81) }*+{\circ}="c"
!{(-.95,.31)}*+{\circ}="d"
!{(0,1)}*+{\bullet}="e"
"a"-"b" "a"-"e"
"b"-"c"
"c"-"d"
"d"-"e"
}\vspace{.1in}
\par
\hspace{0in}\xygraph{
!{<0cm,0cm>;<1cm,0cm>:<0cm,1cm>::}
!{(1,0)}*+{\bullet}="b"
!{(.5,.87)}*+{\bullet}="c"
!{(-.5,.87)}*+{\bullet_v}="d"
!{(-1,0)}*+{\bullet}="e"
!{(-.5,-.87)}*+{\bullet}="f"
!{(.5,-.87)}*+{\circ}="g"
"b"-"c"
"c"-"d"
"d"-"e"
"e"-"f"
"f"-"g"
"g"-"b"
}\quad
\xygraph{
!{<0cm,0cm>;<1cm,0cm>:<0cm,1cm>::}
!{(0,0) }*+{\bullet}="a"
!{(1,0) }*+{\bullet_v}="b"
!{(2,0) }*+{\bullet}="c"
!{(3,0) }*+{\circ}="d"
"a"-"b" "b"-"c" "c"-"d"
}\quad
\xygraph{
!{<0cm,0cm>;<1cm,0cm>:<0cm,1cm>::}
!{(0,-.5) }*+{\bullet}="a"
!{(1,-.5) }*+{\bullet}="b"
!{(2,-.5) }*+{\bullet}="c"
!{(0,.5) }*+{\circ}="d"
!{(1,.5) }*+{\circ}="e"
!{(2,.5) }*+{\bullet_{v}}="f"
"a"-"d" "a"-"e" "a"-"f"
"b"-"d" "b"-"e" "b"-"f"
"c"-"d" "c"-"e" "c"-"f"
}\end{center}}
}
\end{center}
\end{ex}
\vspace{.2cm}
\begin{definition}[Global Antipode]
If $G$ is a graph, the global antipode of $G$ is the subset of $V(G)$ given by
\[
\unig := \cup_{x \in V(G)} \antgx
\]
\end{definition}
\begin{ex}
In the graphs below, the vertices in the global antipode are marked by empty circles:\\

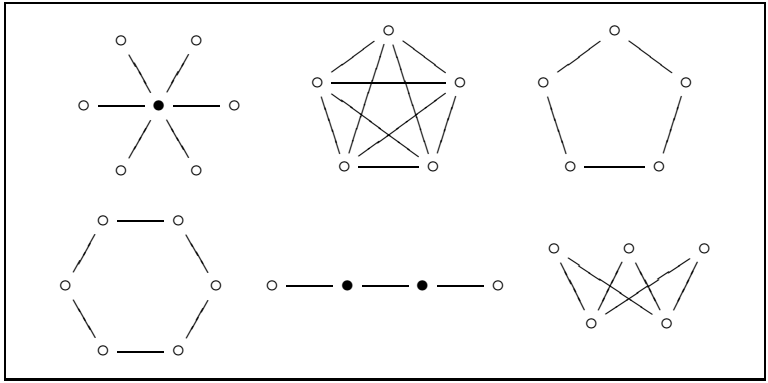
\begin{center}
\framebox[.8\textwidth]{
\parbox{.8\textwidth}{
\begin{center}
\hspace{0in}\xygraph{
!{<0cm,0cm>;<1cm,0cm>:<0cm,1cm>::}
!{(0,0) }*+{\bullet}="a"
!{(1,0) }*+{\circ}="b"
!{(.5,.87) }*+{\circ}="c"
!{(-.5,.87)}*+{\circ}="d"
!{(-1,0)}*+{\circ}="e"
!{(-.5,-.87)}*+{\circ}="f"
!{(.5,-.87)}*+{\circ}="g"
"a"-"b" "a"-"c"
"a"-"d" "a"-"e"
"a"-"f" "a"-"g"
}\qquad
\xygraph{
!{<0cm,0cm>;<1cm,0cm>:<0cm,1cm>::}
!{(.95,.31) }*+{\circ}="a"
!{(.59,-.81) }*+{\circ}="b"
!{(-.59,-.81) }*+{\circ}="c"
!{(-.95,.31)}*+{\circ}="d"
!{(0,1)}*+{\circ}="e"
"a"-"b" "a"-"c" "a"-"d" "a"-"e"
"b"-"c" "b"-"d" "b"-"e"
"c"-"d" "c"-"e"
"d"-"e"
}\qquad
\xygraph{
!{<0cm,0cm>;<1cm,0cm>:<0cm,1cm>::}
!{(.95,.31) }*+{\circ}="a"
!{(.59,-.81) }*+{\circ}="b"
!{(-.59,-.81) }*+{\circ}="c"
!{(-.95,.31)}*+{\circ}="d"
!{(0,1)}*+{\circ}="e"
"a"-"b" "a"-"e"
"b"-"c"
"c"-"d"
"d"-"e"
}\vspace{.1in}
\par
\hspace{0in}\xygraph{
!{<0cm,0cm>;<1cm,0cm>:<0cm,1cm>::}
!{(1,0)}*+{\circ}="b"
!{(.5,.87)}*+{\circ}="c"
!{(-.5,.87)}*+{\circ}="d"
!{(-1,0)}*+{\circ}="e"
!{(-.5,-.87)}*+{\circ}="f"
!{(.5,-.87)}*+{\circ}="g"
"b"-"c"
"c"-"d"
"d"-"e"
"e"-"f"
"f"-"g"
"g"-"b"
}\quad
\xygraph{
!{<0cm,0cm>;<1cm,0cm>:<0cm,1cm>::}
!{(0,0) }*+{\circ}="a"
!{(1,0) }*+{\bullet}="b"
!{(2,0) }*+{\bullet}="c"
!{(3,0) }*+{\circ}="d"
"a"-"b" "b"-"c" "c"-"d"
}\quad
\xygraph{
!{<0cm,0cm>;<1cm,0cm>:<0cm,1cm>::}
!{(.5,-.5) }*+{\circ}="a"
!{(1.5,-.5) }*+{\circ}="b"
!{(0,.5) }*+{\circ}="d"
!{(1,.5) }*+{\circ}="e"
!{(2,.5) }*+{\circ}="f"
"a"-"d" "a"-"e" "a"-"f"
"b"-"d" "b"-"e" "b"-"f"
}\end{center}}
}
\end{center}
\end{ex}
\vspace{.2cm}

Notice that for four of the graphs above, the global antipode is the entire vertex set. However, for the star and the path, this is not the case, leading to a key definition.

\begin{definition}[Global Antipode Graph]\label{GAG}
We call a graph $G$ a global antipode graph (hereafter GAG) if its global antipode is a proper subset of $V(G)$, i.e. $\unig \subsetneq V(G)$.
\end{definition}

For a tree $T$, it is straightforward to verify that for any vertex $t \in V(T)$, $\ant_T(t) \subset \leaf(T)$. Therefore any tree $T$ is a GAG and $\ant(T) \subset \leaf(T) \subsetneq V(T)$. We will discuss examples of GAGs with pathological properties in \S \ref{examples}.

\begin{lemma}\label{antpowers}
If $G$ and $H$ are GAGs, then the Cartesian product $G \cart H$ and the Cartesian power $G^N$ are also GAGs. Moreover, $\ant(G \cart H) = \unig \times \unih$ and $\ant(\powg) = \unig^N$.
\end{lemma}

\begin{proof}
Given GAGs $G$ and $H$, it is suffices to prove only the claims concerning $\cartgh$ because the claims concerning $\powg$ immediately follow by induction. Furthermore, it suffices to prove that $\ant(\cartgh) = \ant(G) \times \ant(H)$ because, following Definition \ref{GAG}, if $\unig \subsetneq V(G)$ and $\unih \subsetneq V(H)$, then trivially $\unig \times \unih \subsetneq V(\cartgh)$.

Recall from Remark \ref{distancerelation} that
\[
d_{G \cart H}\big((g,h),(g',h')\big) = d_G(g,g') + d_H(h,h').
\]
Thus, if $a := \max \{r \geq 0 : \sph_r^G(g) \neq \emptyset\}$ and $b := \max \{r \geq 0 : \sph_r^H(h) \neq \emptyset\}$, then
\[\aligned
\sph_{a+b}^{\cartgh}(g,h) &= \sph_a^G(g) \times \sph_b^H(h),\\
\sph_{a+c}^{\cartgh}(g,h) &= \emptyset \text{ for } c > b, \text{ and}\\
\sph_{d+b}^{\cartgh}(g,h) &= \emptyset \text{ for } d > a.
\endaligned\]
The upshot of the computation is that
\[\aligned
\sph_{r}^{\cartgh}(g,h) &= \antg(g) \times \anth(h) \text{ for } r = a + b \text{ and}\\
\sph_{r}^{\cartgh}(g,h) &= \emptyset \text{ for } r > a + b
\endaligned\]
so $\ant_{\cartgh}\big((g,h)\big) = \antg(g) \times \anth(h)$. Finally we observe that
\[\aligned
\ant(\cartgh) &= \cup_{g \in G, h \in H} \ant_{\cartgh}\big((g,h)\big)\\
&= \cup_{g \in G} \cup_{h \in H} \antg(g) \times \anth(h)\\
&= \big(\cup_{g \in G} \antg(g)\big) \times \big(\cup_{h \in H} \anth(h)\big)\\
&= \unig \times \unih.
\endaligned\]
\end{proof}

The remainder of this section is dedicated to proving a dimensionally exponential lower bound for the $L^p$ norm for the spherical maximal operator on a GAG.

\begin{lemma}\label{identicallyone}
If $G$ is a GAG, then $M^G 1_{\unig}$ is identically equal to $1$.
\end{lemma}

\begin{proof}
For any $x \in V(G)$, by definition $\antgx \subset \unig$. Thus, because $\antgx$ is a nonempty sphere centered at $x$,
\[
M^G 1_{\unig}(x) \geq \frac{1}{|\antgx|} \ee|\sum_{\antgx} 1_{\unig}(y)\ii| = 1
\]

The opposite inequality, $M^G 1_{\unig}(x) \leq 1$, is nothing but the observation that averaging operators cannot increase maxima, i.e. $\|M^G\|_\infty = 1$.
\end{proof}

The main result follows quickly.

\begin{theorem}\label{GAGbound}
If $G$ is a GAG and $M_N$ is the spherical maximal operator on $G^N$, then for all $p<\infty$, $M_N$ satisfies the an exponential lower bound in the $L^p$ operator norm, namely
\[
\ee\|M_N\ii\|^p_p \geq \ee(\frac{|G|}{|\unig|}\ii)^N.
\]
\end{theorem}

\begin{proof}
By Lemma \ref{antpowers}, $\ant(G^N) = \unig^N$ so $|\ant(G^N)| = |\unig|^N$. Thus, $\|1_{\ant(G^N)}\|_p^p =  |\unig|^N$ and by Lemma \ref{identicallyone},
\[
\|M_N 1_{\ant(G^N)}\|_p^p = \|1\|_p^p = |G^N| = |G|^N
\]
Therefore we immediately have the desired lower bound:
\[
\|M_N\|_p^p \geq = \ee|\frac{\|M_N 1_{\ant(G^N)}\|_p}{\|1_{\ant(G^N)}\|_p}\ii|^p = \ee(\frac{|G|}{|\unig|}\ii)^N.
\]
\end{proof}

Because any tree $T$ is a GAG and $\ant(T) \subset \leaf(T)$ (in particular $|\ant(T)| \leq |\leaf(T)|$), Theorem \ref{treebound} follows immediately from Theorem \ref{GAG}.

%%%%%%%%%%%%%%%%%%%%%%%%%%%%

\section{Pathological Examples and Open Questions}\label{examples}

In all graph diagrams in this section, the vertices in the global antipode will be marked by empty circles. Moreover, unless otherwise specified, a ``minimal'' graph with respect to a property $P$ is a graph $G$ with property $P$ such that any other graph with property $P$ has at least as many vertices \emph{and} at least as many edges as $G$. A ``strongly minimal'' graph with respect to a property $P$ is a graph $G$ with property $P$ such that any other graph with property $P$ has strictly more vertices \emph{and} strictly more edges than $G$. We recall the following standard notation:

\begin{notation}
For a graph $G$, $\epsilon(x) = \epsilon_G(x)$ (we generally suppress the subscript) is the eccentricity of $x \in V(G)$. That is,
\[
\epsilon(x) = \max\{r: \sgrx \neq \emptyset\}.
\]
Also, $\radg$ and $\diamg$ are the radius and diameter of $G$ respectively. That is,
\[
\radg = \min_{x \in G} \eccx \leq \max_{x \in G} \eccx = \diamg.
\]
\end{notation}

Notice that $\antgx$ is the sphere of radius $\eccx$ centered at $x$.

\begin{remark}
The proofs in this section will mostly use only elementary graph theory and are not intended to be direct contributions to the overarching project of characterizing asymptotic maximal bounds for Cartesian powers of finite graphs. Rather, they are presented to provide intuition for the relationship between structure and asymptotic maximal bounds. For this reason we present intuitive proofs rather than simply computing minimal examples.
\end{remark}

Based on the notions of vertex eccentricity, radius, and diameter, we identify two classes of graphs that will likely be important in further explorations of maximal bounds for Cartesian powers. Then we will immediately establish their relationship to GAGs.

\begin{definition}[Eccentric Graph]
We call a graph $G$ an eccentric graph (or EG) if not all vertces have the same eccentricity. More concisely, $G$ is eccentric if and only if $\radg < \diamg$.
\end{definition}

\begin{definition}[Sphere Regular Graph]
We call a graph $G$ a sphere regular graph (or SRG) if the size of a sphere depends only on its radius, in particular, not on its center. In order words, for any $r$ and $x,y \in V(G)$, $|\sgrx| = |\sgr(y)|$.
\end{definition}

Notice that if $G$ is an EG, there exist $r$ and $x,y \in V(G)$ such that $\sgrx = \emptyset \neq \sgr(y)$so that $|\sgrx| = 0 < |\sgr(y)|$. Thus EGs are never SRGs.

\begin{lemma}\label{gagsareeccentric}
GAGs are eccentric.
\end{lemma}

\begin{proof}
Let $G$ be a GAG, $x \in V(G) \setminus \unig$, and $y \in \antgx$. By the definition of an antipode, $d(x,y) = \epsilon(y)$ if and only if $x \in \ant_G(y)$. Because $\ant_G(y) \subset \unig$, $x \not\in \ant_G(y)$ so we have the inequality $d(x,y) < \epsilon(y)$.
However, because $y \in \antgx$, $\radg \leq \eccx = d(x,y) < \epsilon(y) \leq \diamg$.
\end{proof}

\begin{cor}\label{gagsarebad}
GAGs cannot be sphere regular, distance regular, vertex transitive, or Cayley graphs.
\end{cor}

\begin{proof}
It is straightforward to verify from definitions that distance regular and vertex transitive graphs are sphere regular. Because eccentric graphs cannot be sphere regular, by Lemma \ref{gagsareeccentric}, GAGs cannot be SRGs. Because Cayley graphs are vertex transitive, the entire corollary follows.
\end{proof}

%%%%%%%%

Although we introduced GAGs as a generalization of trees, this was only because trees are natural and common examples. Heuristically, GAGs are similar to trees in that some of their vertices are more central (analogous to a root) while some are more peripheral (analogous to leaves). It is easily verified that the GAG is the path $P_3$:

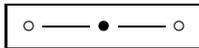
\[ \fbox{\xygraph{
!{<0cm,0cm>;<1cm,0cm>:<0cm,1cm>::}
!{(0,1) }*+{\circ}="a"
!{(1,1) }*+{\bullet}="b"
!{(2,1) }*+{\circ}="c"
"a"-"b" "b"-"c"
} } \]
\vspace{.2cm}

However, the class of GAGs is strictly larger than the class of trees.

\begin{proposition}\label{kitehourglass}
The kite graph (left) is a strongly minimal non-tree GAG and the hourglass graph (right) is a minimal GAG with no leaf vertices (strictly minimal with respect to edges):

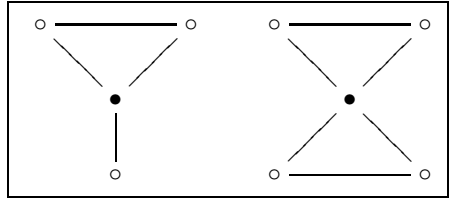
\[ \fbox{ \xygraph{
!{<0cm,0cm>;<1cm,0cm>:<0cm,1cm>::}
!{(0,2) }*+{\circ}="a"
!{(2,2) }*+{\circ}="b"
!{(1,1) }*+{\bullet}="c"
!{(1,0)}*+{\circ}="d"
"a"-"b" "a"-"c"
"b"-"c"
"c"-"d"
}\hspace{.6cm}
\xygraph{
!{<0cm,0cm>;<1cm,0cm>:<0cm,1cm>::}
!{(0,2) }*+{\circ}="a"
!{(2,2) }*+{\circ}="b"
!{(1,1) }*+{\bullet}="c"
!{(0,0)}*+{\circ}="d"
!{(2,0)}*+{\circ}="e"
"a"-"b" "a"-"c"
"b"-"c" "c"-"d"
"d"-"e" "c"-"e"
} } \]
\end{proposition}
\vspace{.2cm}

\begin{proof}
Calculating global antipodes is straightforward so we will only prove minimality.

Kite: The only $3$ vertex graphs are are $P_3$ and $K_3$ so we first prove that there are no $4$ vertex non-tree GAGs other than the kite graph above. To this end, suppose $G$ is non-tree GAG with $4$ vertices. If there is no vertex of degree $3$, $G$ is a path and thus a tree so the will let $1$ be a vertex with degree $3$. Notice that $\antgx$ consists of all other vertices because $\eccx = 1$. Letting $0$, $2$, and $3$ denote the remaining $3$ vertices, there must be at least one edge between $0$, $2$, and $3$ or else $G$ would be a tree. Without loss of generality $2$ and $3$ are adjacent:

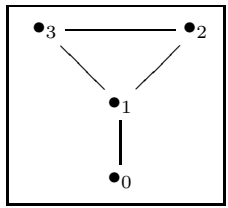
\[ \fbox{ \xygraph{
!{<0cm,0cm>;<1cm,0cm>:<0cm,1cm>::}
!{(0,2) }*+{\bullet_{3}}="a"
!{(2,2) }*+{\bullet_{2}}="b"
!{(1,1) }*+{\bullet_{1}}="c"
!{(1,0)}*+{\bullet_{0}}="d"
"a"-"b" "a"-"c"
"b"-"c"
"c"-"d"
} } \]
\vspace{.2cm}

If $2$ or $3$ is adjacent to $0$ (without loss of generality $2$ is adjacent to $0$), then $\epsilon(2) = 1$ so
\[
V(G) = \antg(1) \cup \antg(2) \subset \unig
\]
and thus $G$ is not a GAG. Therefore the kite is the only GAG with $4$ vertices. It follows that there are no other non-tree GAGs with $4$ edges because, if $G$ were such a graph, it would have $4$ edges and at least $5$ vertices so it could only be $P_5$, a tree.\\

Hourglass: Suppose $G$ is a GAG with no vertices of degree $1$ (i.e. leaves). Because $G$ is not a tree and it cannot be the kite graph, we just showed it must have at least $5$ vertices and at least $5$ edges. Suppose $G$ has $5$ vertices. Because each vertex has degree at least $2$, any pair of non-adjacent vertices must share a neighbor simply by the pigeonhole principle so $\diamg = 2$. By Corollary \ref{gagsareeccentric}, $\radg = 1$ so there exists a vertex, which we call $2$, connected to all other vertices. Moreover, because no vertices have degree $1$, there must be at least (actually strictly more than) one edge connecting two of the remaining vertices so without loss of generality we assume vertices $1$ and $2$ are adjacent.  Therefore $G$ contains the following subgraph:

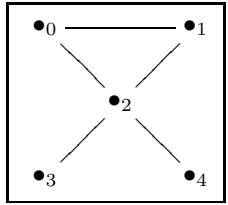
\[ \fbox{ \xygraph{
!{<0cm,0cm>;<1cm,0cm>:<0cm,1cm>::}
!{(0,2) }*+{\bullet_0}="a"
!{(2,2) }*+{\bullet_1}="b"
!{(1,1) }*+{\bullet_2}="c"
!{(0,0)}*+{\bullet_3}="d"
!{(2,0)}*+{\bullet_4}="e"
"a"-"b" "a"-"c" "b"-"c" "c"-"d" "c"-"e"
} } \]
\vspace{0cm}

The only graph with at most one extra edge containing this subgraph is the hourglass, thus proving that all other GAGs without a vertex of degree $1$ have more edges or more vertices. Notice that the only graphs with more than $5$ vertices that have at most $6$ edges are either trees (and thus have vertices of degree $1$) or the cycle $C_6$. Because cycles are vertex transitive, by Corollary \ref{gagsarebad}, $C_6$ is not a GAG. Thus there cannot be another GAG with $6$ edges and no leaves.
\end{proof}

We observe also that the global antipode of a tree may be strictly smaller than its leaf set.

\begin{proposition}
The graph below is the strongly minimal tree whose leaf set is strictly larger than its global antipode:
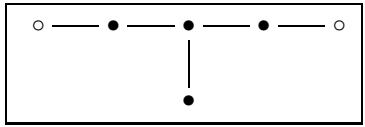
\[ \fbox{ \xygraph{
!{<0cm,0cm>;<1cm,0cm>:<0cm,1cm>::}
!{(0,1) }*+{\circ}="a"
!{(1,1) }*+{\bullet}="b"
!{(2,1) }*+{\bullet}="c"
!{(3,1)}*+{\bullet}="d"
!{(4,1)}*+{\circ}="e"
!{(2,0)}*+{\bullet}="f"
"a"-"b" "b"-"c"
"c"-"d" "d"-"e"
"c"-"f"
} } \]
\end{proposition}
\vspace{0cm}
\begin{proof}
For path and star graphs, the antipode of any leaf is precisely the set of all other leaves. Thus the global antipode must include all leaves, so if $G$ is strongly minimal among trees with the property that $\leaf(G) \neq \unig$, it can be neither a star nor a path.

Because it is not a path it must have a vertex of degree at least $3$ and, because it is not a star, there must be more than $3$ other vertices so $|V(G)| \geq 5$. If $G$ has $5$ vertices, its vertices must have degree at most $3$ as it is not a star. The only $5$ vertex tree with a vertex of degree $3$ is the following:
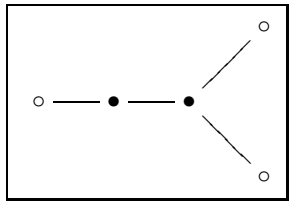
\[ \fbox{ \xygraph{
!{<0cm,0cm>;<1cm,0cm>:<0cm,1cm>::}
!{(0,1)}*+{\circ}="a"
!{(1,1)}*+{\bullet}="b"
!{(2,1)}*+{\bullet}="c"
!{(3,0)}*+{\circ}="d"
!{(3,2)}*+{\circ}="e"
"a"-"b" "b"-"c"
"c"-"d" "c"-"e"
} } \]
\vspace{0cm}

Because the leaf set and the global antipode of the above graph are equal, it must be the case that $|V(G)| \geq 6$. If $G$ has $6$ vertices, it cannot have a vertex of degree $5$ as it would be a star. Moreover, the only tree with a vertex of degree $4$ is the following:
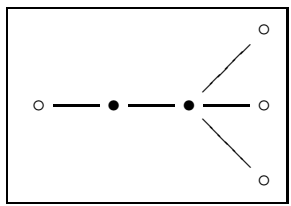
\[ \fbox{ \xygraph{
!{<0cm,0cm>;<1cm,0cm>:<0cm,1cm>::}
!{(0,1)}*+{\circ}="a"
!{(1,1)}*+{\bullet}="b"
!{(2,1)}*+{\bullet}="c"
!{(3,0)}*+{\circ}="d"
!{(3,1)}*+{\circ}="e"
!{(3,2)}*+{\circ}="f"
"a"-"b" "b"-"c"
"c"-"d" "c"-"e" "c"-"f"
} } \]
\vspace{.2cm}

Therefore the highest vertex degree in $G$ must be $3$ so it contains the following (disconnected) subgraph:
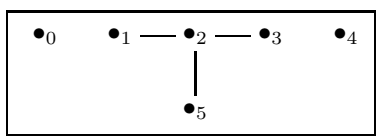
\[ \fbox{ \xygraph{
!{<0cm,0cm>;<1cm,0cm>:<0cm,1cm>::}
!{(0,1)}*+{\bullet_0}="a"
!{(1,1)}*+{\bullet_1}="b"
!{(2,1)}*+{\bullet_2}="c"
!{(3,1)}*+{\bullet_3}="d"
!{(4,1)}*+{\bullet_4}="e"
!{(2,0)}*+{\bullet_5}="f"
"b"-"c" "c"-"d" "c"-"f"
} } \]
\vspace{.2cm}

At least one of vertices $0$ and $4$, without loss of generality $4$, must connect to the central connected component containing $1$, $2$, $3$, and $5$. Without loss of generality $4$ connects to $3$. Connecting $0$ to one of the remaining vertices yields a graph isomorphic to one of the following:

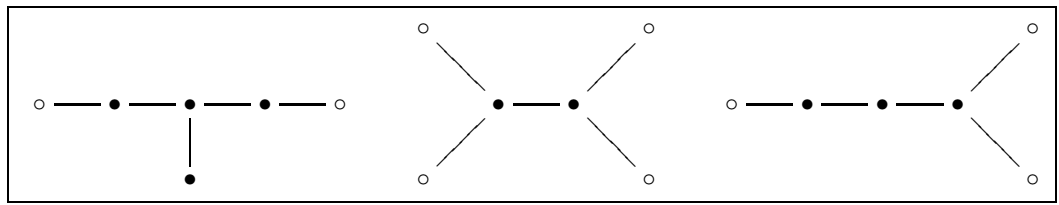
\[ \fbox{ \xygraph{
!{<0cm,0cm>;<1cm,0cm>:<0cm,1cm>::}
!{(0,1)}*+{\circ}="a"
!{(1,1)}*+{\bullet}="b"
!{(2,1)}*+{\bullet}="c"
!{(3,1)}*+{\bullet}="d"
!{(4,1)}*+{\circ}="e"
!{(2,0)}*+{\bullet}="f"
"a"-"b" "b"-"c"
"c"-"d" "d"-"e"
"c"-"f"
}\hspace{.6cm}
\xygraph{
!{<0cm,0cm>;<1cm,0cm>:<0cm,1cm>::}
!{(0,0)}*+{\circ}="a"
!{(0,2)}*+{\circ}="b"
!{(1,1)}*+{\bullet}="c"
!{(2,1)}*+{\bullet}="d"
!{(3,0)}*+{\circ}="e"
!{(3,2)}*+{\circ}="f"
"c"-"b" "a"-"c"
"c"-"d"
"d"-"e" "d"-"f"
}\hspace{.6cm}
\xygraph{
!{<0cm,0cm>;<1cm,0cm>:<0cm,1cm>::}
!{(0,1)}*+{\circ}="a"
!{(1,1)}*+{\bullet}="b"
!{(2,1)}*+{\bullet}="c"
!{(3,1)}*+{\bullet}="d"
!{(4,0)}*+{\circ}="e"
!{(4,2)}*+{\circ}="f"
"a"-"b" "b"-"c"
"c"-"d" "d"-"e" "d"-"f"
} } \]
\vspace{.2cm}

Because only the leftmost graph has a leaf outside its global antipode, this proves the claim.
\end{proof}

%%%%%%%%%

Although GAGs cannot be sphere regular, they can in fact be regular.

\begin{proposition}\label{regulargag}
All regular GAGs have at least $10$ vertices. Moreover, there are multiple non-isomorphic examples that are minimal in both vertices and edges among regular GAGs.
\end{proposition}

\begin{proof}
It is straightforward to verify that the following graphs are $3$-regular GAGs with $10$ vertices:
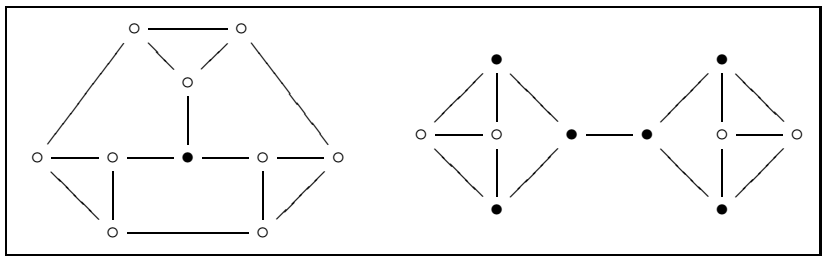
\[ \fbox{ \xygraph{
!{<0cm,0cm>;<1cm,0cm>:<0cm,1cm>::}
!{(0,0) }*+{\bullet}="a"
!{(-1,0) }*+{\circ}="b"
!{(1,0) }*+{\circ}="c"
!{(0,1)}*+{\circ}="d"
!{(-2,0)}*+{\circ}="e"
!{(-1,-1)}*+{\circ}="f"
!{(2,0)}*+{\circ}="g"
!{(1,-1)}*+{\circ}="h"
!{(-.71,1.71)}*+{\circ}="i"
!{(.71,1.71)}*+{\circ}="j"
"a"-"b" "a"-"c" "a"-"d"
"d"-"i" "d"-"j" 
"b"-"e" "b"-"f"
"c"-"h" "c"-"g"
"e"-"f" "f"-"h" "h"-"g" "g"-"j" "j"-"i" "i"-"e"
}\hspace{.6cm}
\xygraph{
!{<0cm,0cm>;<1cm,0cm>:<0cm,1cm>::}
!{(0,.3) }*+{\circ}="a"
!{(-1,.3) }*+{\circ}="b"
!{(1,.3) }*+{\bullet}="c"
!{(0,1.3)}*+{\bullet}="d"
!{(0,-.7)}*+{\bullet}="e"
!{(3,.3)}*+{\circ}="f"
!{(2,.3)}*+{\bullet}="g"
!{(4,.3)}*+{\circ}="h"
!{(3,1.3)}*+{\bullet}="i"
!{(3,-.7)}*+{\bullet}="j"
"a"-"b" "a"-"d" "a"-"e"
"b"-"e" "b"-"d"
"c"-"d" "c"-"e" "c"-"g"
"f"-"i" "f"-"h" "f"-"j"
"g"-"i" "g"-"j"
"h"-"i" "h"-"j"
} } \]
\vspace{.2cm}

Moreover, they are certainly not isomorphic, as seen by noting the graph on the left has a $9$ element global antipode while that on the right has a $4$ element global antipode. Alternatively, one can note that the graph on the left has radius $2$ and diameter $3$ while that on the right has radius $3$ and diameter $5$.

We will now show that there cannot exist regular GAGs with fewer than $10$ vertices. Notice first that any $2$-regular graph is a cycle and, as in Proposition \ref{kitehourglass}, cannot be a GAG. This also means that once we show there are no regular GAGs with fewer than $10$ vertices, the $3$-regular $10$ vertex GAGs above must have minimal edges among regular GAGs. 

Let $G$ be a regular GAG with a minimal size vertex set among regular GAGs. Because we know $G$ must be at least $3$-regular, immediately we know $|V(G)| \geq 4$. Because a complete graph is vertex transtive, $G$ cannot be $K_4$. If $|V(G)| = 5$, then $G$ must be at least $4$-regular because the sum of its vertices' degrees must be even. However this cannot hold because $G$ cannot be $K_5$.

Following this line of reasoning, if $|G(V)| = 6$, $G$ must be $3$ or $4$ regular. Either way, by the pigeonhole principle, if $x, y \in V(G)$ are not connected, they must share a neighbor so for any $x, y \in V(G)$, $d(x,y) \leq 2$. Therefore $\diamg \leq 2$. However, because $G$ is $3$ or $4$ regular, no vertex is connected to every other vertex and so $\radg \geq 2 \geq \diamg \geq \radg$ so $G$ is not eccentric. Thus, by Lemma \ref{gagsareeccentric}, $G$ cannot be a GAG.

Because $7$ is odd, if $|V(G)| = 7$, $G$ must be $4$ or $6$ regular. $6$ is impossible because $G$ is not complete and the same pigeonhole argument for the $|G(V)| = 6$ case above shows that $G$ is not eccentric if it is $4$-regular so it cannot be a GAG. Similarly if $|V(G)| = 9$, the same argument shows that it cannot be $4$, $6$, or $8$ regular. Once again, if $|V(G)| = 8$ then $G$ cannot be $4$ (or more) regular because this would prevent it from being eccentric. Thus, as long as $G$ is not $3$-regular with $8$ vertices, it must contain at least $10$ vertices as claimed.

Suppose $G$ is $3$-regular and $|V(G)| = 8$. As before, $\diamg \geq 3$ so there exist vertices $0$ and $1$ such that $d_G(0,1) \geq 3$. Because $0$ and $1$ cannot share a neighbor, by the pigeonhole principle $G$ must contain the following (disconnected) subgraph:

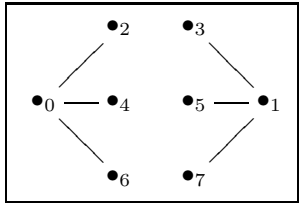
\[ \fbox{ \xygraph{
!{<0cm,0cm>;<1cm,0cm>:<0cm,1cm>::}
!{(-1.5,0) }*+{\bullet_0}="w"
!{(-.5,1) }*+{\bullet_2}="x"
!{(-.5,0) }*+{\bullet_4}="y"
!{(-.5,-1)}*+{\bullet_6}="z"
!{(1.5,0)}*+{\bullet_1}="a"
!{(.5,1)}*+{\bullet_3}="b"
!{(.5,0)}*+{\bullet_5}="c"
!{(.5,-1)}*+{\bullet_7}="d"
"a"-"b" "a"-"c" "a"-"d"
"w"-"x" "w"-"y" "w"-"z"
 } } \]
\vspace{.2cm}

We will refer to the vertices with odd labels as ``odd vertices'' and even labels as ``even vertices.'' Because $G$ is connected, there must be an edge between an even and an odd vertex and thus, $d_G(0,1)=3$. Because this holds for any pair of vertices that share no neighbors (in place of $0$ and $1$), $\diamg = 3$. Therefore any vertex with eccentricity $3$ is in the global antipode (in particular, $0$ and $1$).

Notice that the (possibly disconnected) subgraph $S$ of $G$ induced by the vertex set $V := \{2,3,4,5,6,7\}$ is $2$-regular because $G$ is $3$ regular and each vertex in $V$ connects to exactly one vertex outside of $V$. Thus $S$ is either composed of two $3$-cycles or one $6$-cycle.

Because $G$ is a GAG, there must be a vertex, without loss of generality $2$, that is not in $\unig$. However, if $S$ is composed of two $3$-cycles, there must exist at least one odd vertex, without loss of generality $7$, that is not in the same $3$-cycle as $2$ within $S$. This means $2$ and $7$ do not share a neighbor so $\epsilon_G(2) \geq 3 = \diamg$ so $2 \in \unig$, a contradiction.

This leaves only the possibility that $S$ consists of a single $6$-cycle and we can therefore partition $V$ into pairs that are mutually antipodal in $S$. Suppose that the antipode of $2$ in $S$ is an odd vertex, without loss of generality $1$ so $d_S(2,3) = 3$. Then any path between $2$ and $3$ in $G$ that is not contained in $S$ must contain either $0$ or $1$ and will therefore have length at least $3$ because $d_G(2,1) = 2$ and $d_G(0,3) \geq 2$. Therefore the antipode of $2$ in $S$ is even. Without loss of generality, it is $4$.

Notice that $d_G(2,4) = 2$ because they do not share an edge in $E(G) \setminus E(S)$ and, by merit of being an antipodal pair in $S$, $d_S(2,4) = 3$. Because $2 \not\in \unig$ and thus $2 \not\in \antg(4)$, there must be a vertex in $V(G)$ that is distance $3$ (in $G$) away from vertex $4$. This cannot be an even vertex because the even vertices are all connected to $0$ or $0$ itself. Moreover, because $2$ is the unique antipodal vertex of $4$ in $S$, vertices $3$, $5$, and $7$ are no further than $2$ from $4$ in $S$ and therefore in the larger graph $G$ as well. This only leaves the possibility that $d_G(4,1) = 3$.

However, if $d_G(4,1) = 3$, $4$ cannot share a neighbor with $1$ so it cannot have an odd neighbor. Thus, by the pigeonhole principle, the neighbors of $4$ are $0$, $2$, and $6$. In particular, this implies that $d_G(2,4) = 1$, contradicting our earlier observation that $d_G(2,4) = 2$. 
\end{proof}

%%%%%%%%%%%%%%%%%%%
The exponential lower bound of Theorem \ref{GAGbound} is not sharp, however the argument is a clean and relatively simple illustration of bad dimensional asymptotics for maximal bounds. Moreover, we have no reason to think that GAGs are the most general class of graphs with exponentially growing maximal bounds. Even if the result is generalized, however, the argument will most likely be more intricate and technical. Therefore the argument presented here for GAGs is pedagocally ideal.

We conclude by bringing up two currently open questions following from Theorem \ref{GAGbound}. The gap between graphs known to have dimension-independent maximal bounds (at this point, all finite cliques) and GAGs is wide. The author is currently working to bridge that gap by exploring what the dimensionally asymptotic bounds look like for graphs that are not as well behaved as cliques but not as poorly behaved as GAGs.

For instance, it is straightforward to verify that Cartesian powers of an SRG are also SRGs and that spherical averaging operators are $L^1$ contractions for a graph if and only if the graph is sphere regular. As in Remark \ref{distancerelation}, for any graph $G$, $\diam(G^N) = N \diamg$ and thus there are $N \diamg + 1$ spherical averaging operators on $G^N$. Thus, if $G$ is sphere regular and $M_N$ is the spherical maximal operator on $G^N$, then $M_N$ is the maximum of $N \diamg + 1$ linear operators bounded by $1$ in $L^1(G^N)$.

By the triangle inequality, $\|M_N\|_1 \leq N \diamg + 1$. Linearizing the maximal operator as in \cite{NS} (top of p. 151) and interpolating between the $L^1$ and $L^\infty$ endpoints provides the bound $\|M_N\|_p \leq (N \diamg+1)^{1/p} \sim N^{1/p}$ for all $p \leq \infty$. Therefore sphere regular graphs' maximal bounds grow far slower than exponentially in dimension, even if they are unbounded. This gives rise to a natural question:

\begin{question}\label{generality}
Under what conditions on a graph do the sphercal maximal $L^p$ bounds grow exponentially in dimension?
\end{question}

At the moment we suspect that a graph has exponentially growing bounds if and only if it is not sphere regular. It may be easier to prove the result for EGs before SRGs or it may turn out to only be true for EGs.

Another interesting follow-up question comes from an observation about our proof of Theorem \ref{GAGbound}. Notice that we relied on the maximal radius sphere centered at a given point to provide a large average. If we muted the effect of distant spheres, this proof falls apart and it is not clear if the bounds are still exponentially growing. The most natural way to formalize the notion of muting the effects of distant spheres is by examining the ball maximal operator in place of the sphere maximal operator (i.e. $Mf(x)$ is the maximal magnitude average over balls centered at $x$ rather than spheres centered at $x$).

\begin{question}\label{balls}
What is the dimensional growth rate for the ball maximal operator on Cartesian powers of GAGs?
\end{question}

\end{document}